\documentclass{amsart}
\usepackage{amsmath,amssymb,enumerate}
\usepackage[mathscr]{eucal}
\usepackage[polish,english]{babel}
\usepackage[cp1250]{inputenc}
\usepackage[T1]{fontenc}
\usepackage{txfonts}

\newtheorem{theorem}{Theorem}[section]
\newtheorem{proposition}[theorem]{Proposition}
\newtheorem{corollary}[theorem]{Corollary}
\newtheorem{lemma}[theorem]{Lemma}
\theoremstyle{definition}
\newtheorem{definition}[theorem]{Definition}

\newtheorem{problem}[theorem]{Problem}
\numberwithin{equation}{section}
\newtheorem{example}[theorem]{Example}

\begin{document}
\title{On subadditive functions bounded above on~a~`large' set}


\author[N.H. Bingham]{Nicholas H. Bingham}
\address{Department of Mathematics, Imperial College London, South Kensington Campus, London, UK, SW7 2AZ}
              \email{n.bingham@imperial.ac.uk}

\author[E. Jab{\l}o\'{n}ska]{Eliza Jab\l o\'nska}
\address{Department of Mathematics, Pedagogical University of Cracow, Podchor\c{a}\.{z}ych~2, 30--084 Krak\'{o}w, Poland}
\email{eliza.jablonska@up.krakow.pl}
\author[W. Jab{\l}o\'{n}ski]{Wojciech Jab\l o\'{n}ski}
\address{Department of Mathematics, Pedagogical University of Cracow, Podchor\c{a}\.{z}ych~2, 30--084 Krak\'{o}w, Poland}
              \email{wojciech.jablonski@up.krakow.pl}
              \author[A.J. Ostaszewski]{Adam J. Ostaszewski}
\address{Mathematics Department, London School of Economics, Houghton Street, London, UK, WC2A 2AEA}
              \email{a.j.ostaszewski@lse.ac.uk}

\date{}

\begin{abstract}
It is well known that boundedness of a subadditive function need not imply its
continuity. Here we prove that each subadditive function $f:X\rightarrow
\mathbb{R}$ bounded above on a~shift--compact (non--Haar--null, non--Haar--meagre) set is
locally bounded at each point of the domain. Our results refer to \cite[Chapter~16]{Kuczma} and papers by N.H.~Bingham and A.J.~Ostaszewski \cite{BO,BinO1,BinO2,BinO6,BinO7}.
\end{abstract}
\maketitle

\section{Motivation and auxiliary results}

Let $X$ be an abelian topological group. A function $f:X\rightarrow \mathbb{R%
}$ is \textit{subadditive} if
\begin{equation*}
f(x+y)\leq f(x)+f(y)\;\;\mbox{for every}\;\;x,y\in X\,\footnote{\,In the paper we do not admit infinite values of $f$. For more information on infinite--valued subadditive functions see \cite[Chapter~16]{Kuczma}.}.
\end{equation*}
A function $f:X\rightarrow \mathbb{R}$ which is subadditive and also
satisfies
\begin{equation*}
f(nx)=nf(x)\;\;\mbox{for every}\;\;x\in X,\;n\in \mathbb{N}
\end{equation*}
is called \textit{sublinear}.

Subadditive and sublinear functions play a fundamental role in mathematics
and so have attracted the interest of many authors (see e.g.~\cite{Berz},
\cite{BO}, \cite{Cooper}, \cite{Hille}, \cite{Kuczma}, \cite{Matkowski},
\cite{Matkowski1}, \cite{Rosenbaum}). Examples of subadditive functions
include norms, seminorms, and the function $\mathbb{R}\ni x \mapsto \sqrt{|x|}\in\mathbb{R}$ (see e.g. \cite[Theorem~7.2.5]{Hille}).  A classical result
concerning subadditive functions is

\begin{theorem}
[{\cite[Theorem 16.2.2]{Kuczma}}]\label{t1} If a subadditive function $f:\mathbb{R}^{N}\rightarrow \mathbb{R}$ is bounded above locally at some
point, then $f$ is locally bounded at each point of~$\mathbb{R}^{N}$.
\end{theorem}

Note that local boundedness of a subadditive function does not imply its
continuity; any function having values in the interval $[a,2a]$, with $a>0$, is subadditive.

A result stronger than Theorem~\ref{t1} is the following.

\begin{theorem}
[{\cite[Theorem 16.2.4]{Kuczma}}]\label{t2} If a subadditive function $f:
\mathbb{R}^{N}\rightarrow \mathbb{R}$ is bounded above on a set $T\subset
\mathbb{R}^{N}$ and such that its $k$--fold sum
$\sum_{i=1}^kT$ has positive inner Lebesgue measure or is non--meagre for some $k\in \mathbb{N}$, then $f$ is locally bounded at each point of $\mathbb{R}^{N}$.
\end{theorem}

Below we generalize the two theorems above using the notion of a
shift--compact set (see \cite[III.2]{Par}, \cite[5.1]{Hey}).

\begin{definition}
In an abelian topological group $X$, a set $A\subset X$ is called \textbf{shift--compact} if for every sequence $(x_{n})_{n\in \mathbb{N}}$ tending to
$0$ in $X$ there exists $x\in X$ such that the set $\{n\in \mathbb{N}:x+x_{n}\in A\}$ is infinite.
\end{definition}

Shift--compact sets were used in the context of subadditive functions for the
first time by N.H.~Bingham and A.J.~Ostaszewski in \cite{BinO2}. In the
Euclidean context of \cite{BinO1} a~shift--compact set was earlier called
subuniversal (following H.~Kestelman \cite{Kes}), and also null--shift--compact
(in \cite{BinO4} in the special context of $\mathbb{R}$). The underlying
group action is studied in \cite{MilO}.

The notion of a shift--compact set is directly equivalent to the notion of a
\textit{non}--null--finite set. The definition of null-finite sets was introduced by T.~Banakh and E.~Jab{\l}o\'{n}ska in \cite{BJ} (see also \cite{Kwela}). Using \cite[Proposition~2.2]{BJ} we can easily explain the name "shift--compact": a set is shift--compact if and only if infinitely many points of this set belong to a translation of any infinite compact set. It is clear that non--empty open sets are shift--compact. Moreover, in view of \cite[Theorem~3.1~(2)]{BJ}, countable sets in non--discrete metric groups are not shift--compact. The result of A.~Kwela \cite[Theorem~4.1]{Kwela}, that the Cantor set is not shift--compact, seems to be of significant interest.

To indicate the extent of shift--compact sets we need to recall two terms: Haar--null sets, as defined by
J.P.R.~Christensen in \cite{Ch}, and their topological analogue the
Haar--meagre sets, as defined by U.B.~Darji \cite{Darji}.

\begin{definition} Let $X$ be a complete abelian metric group.
A set $A\subset X$ is called \textbf{universally measurable}\,\footnote{\,See e.g. \cite[p. 227]{Kechris}.} if it
is measurable with respect to each complete Borel probability measure on $X$. A~universally measurable set $B\subset X$ is \textbf{Haar--null}
if there exists a $\sigma $--additive probability Borel measure $\mu$ on $X$
such that $\mu (B+x)=0$ for all $x\in X$.
\end{definition}

Christensen \cite{Ch} proves that in each locally compact complete abelian
metric group the notion of a Haar--null set is equivalent to the notion of a
set of Haar measure zero.

\begin{definition} Let $X$ be a complete abelian metric group.
A set $A\subset X$ is called \textbf{universally Baire}\,\footnote{\,See e.g. \cite{Magidor}.} if for each
continuous function $f:K\rightarrow X$ mapping a compact metric space $K$
into $X$ the set $f^{-1}(A+x)$ has the Baire property for every $x\in X$. A~universally Baire\,set $B\subset X$
is called \textbf{Haar--meagre} if
there exists a~continuous map $f:K\rightarrow X$ from a non--empty compact
metric space $K$ such that the set $f^{-1}(B+x)$ is meagre in $K$ for every $x\in X$.
\end{definition}

Darji \cite{Darji} shows that in each locally compact complete abelian
metric group the notion of a Haar--meagre set is equivalent to the notion of
a meagre set. However, in the non--locally compact case, there is only a
one--sided inclusion: Haar--meagre sets are meagre, but the converse
implication may fail.

Armed with these terms, we may now note (see \cite[Theorems 5.1 and 6.1]{BJ}
and also \cite[Theorem 3]{BinO8}) that

\begin{theorem}
\label{HN+HM} In a complete abelian metric group\,\footnote{\,Here by a \textit{metric group} we mean a group with an invariant metric; by the Birkhoff--Kakutani theorem (see e.g. \cite[Theorem~9.1.]{Kechris}) any metrizable topological group is a metric group.}:
\begin{itemize}
\item[{(i)}] each universally Baire non--Haar--meagre set is
shift--compact;
\item[{(ii)}] each universally measurable non--Haar--null set is
shift--compact.
\end{itemize}
\end{theorem}

The converse of Theorem~\ref{HN+HM} does not hold: see \cite[Theorem 12]{CieR}, \cite[Example~7.1]{BJ} and also \cite{MilMO}.

It emerges that in non--locally compact complete abelian metric groups there
exist sets which are neither Haar--null nor Haar--meagre and with $k$--fold
sums that are meagre for each $k\in \mathbb{N}$.  An example is provided by
the positive cone $C:=\{(x_{n})_{n\in \mathbb{N}}\in c_{0}:x_{n}\geq 0\;\mbox{for each}\;n\in \mathbb{N}\}$ in the space $c_{0}$ (of all real
sequences tending to zero); this is neither Haar--meagre nor Haar--null, $C=\sum_{i=1}^kC$ for each $k\in \mathbb{N}$ and $C$ is nowhere dense in $X$ (see
\cite{Jab}, \cite{MZaj}). By the Steinhaus--Pettis--Piccard Theorem (see
\cite{Pet}, \cite{Pic}, \cite{Stein} or \cite[Theorems~2.9.1, 3.7.1]{Kuczma}), such a situation is not possible in the case of locally compact abelian
Polish groups, where the families of Haar--meagre sets and of meagre sets coincide, and likewise the families of Haar--null sets and of sets of Haar measure zero coincide. This motivates the following:

\begin{problem}\label{el}
For $X$ a complete abelian metric group and $f:X\rightarrow \mathbb{R}$
a~subadditive function, bounded above on a set $T\subset X$ with a $k$--fold sum $\sum_{i=1}^kT$ either universally Baire and
non--Haar--meagre, or universally measurable and non--Haar--null, is $f$ locally bounded at each point of $X$?
\end{problem}

Below we give an affirmative answer. Actually, we prove a~more general
result: we show that for $X$ an abelian metric group, every subadditive
function $f:X\rightarrow \mathbb{R}$ that is bounded above on a shift--compact subset of $X$ is necessarily locally bounded at each point
of~$X$.

Analogous results for additive functions as well as mid--point convex
functions (i.e. functions satisfying
\begin{equation*}
f((x+y)/2)\leq (f(x)+f(y))/2
\end{equation*}
for
every $x,y$ from the domain of $f$) were obtained in \cite[Theorems 9.1 and
11.1]{BJ}, and also in \cite[$§$7 and Theorem 1]{BinO1} in the case of $\mathbb{R}$ (cf. \cite[$§$10 Convexity]{BinO2}), where the following two
results were proved.  (We use `shift--compact' in place of `non--null--finite' as in [1]; see \S 1.)

\begin{theorem}
[{\cite[Theorem 9.1]{BJ}}]\label{tBJ} Each additive functional $f:X\rightarrow \mathbb{R}$ on an abelian metric group $X$ is bounded above
on a shift--compact set in $X$ if and only if it is continuous.
\end{theorem}

\begin{theorem}
[{\cite[Theorem 11.1]{BJ}}]\label{tBJ1} Each mid--point convex function $f:X\rightarrow \mathbb{R}$ defined on a real linear metric space $X$ is
bounded above on a shift--compact set in $X$ if and only if it is
continuous.
\end{theorem}

Since each sublinear function $f:X\rightarrow \mathbb{R}$ defined on a real
linear space $X$ is necessarily mid--point convex (see \cite[Lemma 16.1.11]{Kuczma}), from Theorem~\ref{tBJ1} we obtain the

\begin{corollary}
\label{lin} Each sublinear function $f:X\rightarrow \mathbb{R}$ defined on a
real linear metric space $X$ is bounded above on a shift--compact set in $X$ if and only if it is continuous.
\end{corollary}

The result above in the case $X=\mathbb{R}$ was obtained by N.H.~Bingham
and A.J.~Ostaszewski in \cite[Proposition 5]{BinO6} (cf. \cite[Theorem R]{BinO7} and \cite[Proposition 5]{BinO9}).

This is also why boundedness from above of subadditive functions on
shift--compact sets seems to be all the more interesting.

Finally, we determine the relationship between local boundedness at some
point, boundedness from above on a shift--compact set and property WNT, as
proposed by N.H. Bingham and A.J. Ostaszewski in \cite{BO}, in the class of
functions $f:X\rightarrow \mathbb{R}$ defined on an abelian topological
group $X$.

\section{Main results}

First, let us recall some basic properties of subadditive functions.

\begin{lemma}
[{\cite[Lemmas 16.1.3, 16.1.4, 16.1.5]{Kuczma}}]\label{1} Let $X$ be a
group and $f:X\to\mathbb{R}$ be a~subadditive function. Then:
\begin{itemize}
\item[{(i)}] $f(0)\geq 0$;
\item[{(ii)}] $f(-x)\geq -f(x)$ for each $x\in X$.
\end{itemize}
\end{lemma}

We are now ready to prove the main result, which in the case of $\mathbb{R}$, was obtained in \cite[Theorem 2~(ii) and Remark]{BinO1}.

\begin{theorem}
\label{main} Let $X$ be an abelian metric group and $f:X\rightarrow \mathbb{R}$ a~subadditive function. If $f$ is bounded above on a set $T\subset X$
whose $k$--fold sum $\sum_{i=1}^kT$ is shift--compact for some $k\in \mathbb{N}$,
then $f$~is locally bounded at each point of~$X$.
\end{theorem}

In the proof of the above result we use classical methods from \cite[proof of Theorem~16.2.]{Kuczma} and from \cite[proof of Theorem~2~(ii)]{BinO1}.

\begin{proof}
Suppose that $f$ is not locally bounded at the point $x_{0}\in X$. This
means that there is a sequence $(x_{n})_{n\in \mathbb{N}}$ with $x_{n}\rightarrow x_{0}$ and $|f(x_{n})|\rightarrow \infty $. Then we may
choose either a subsequence $(x_{n}^{\prime })_{n\in \mathbb{N}}$ of $(x_{n})_{n\in \mathbb{N}}$ with $f(x_{n}^{\prime })>n$ for each $n\in
\mathbb{N}$ or a~subsequence $(x_{n}^{\prime \prime })_{n\in \mathbb{N}}$ of
$(x_{n})_{n\in \mathbb{N}}$ with $f(x_{n}^{\prime \prime })<-n$ for each $n\in \mathbb{N}$.

In the first case, put $y_{n}:=x_{n}^{\prime }$ for $n\in \mathbb{N}$ and $y_{0}:=x_{0}$. In the second case, take $y_{n}:=-x_{n}^{\prime \prime }$ for
$n\in \mathbb{N}$ and $y_{0}:=-x_{0}$. By Lemma~\ref{1}, $f(-x_{n}^{\prime
\prime })\geq -f(x_{n}^{\prime \prime })>n$ for each $n\in \mathbb{N}$ so,
in both cases, there exists a sequence $y_{n}\rightarrow y_{0}$ such that $f(y_{n})>n$ for $n\in \mathbb{N}$. Since $y_{n}-y_{0}\rightarrow 0$ and $\sum_{i=1}^kT\subset X$ is shift--compact, there exists $z_{0}\in X$ such that the
set $\mathbb{N}_{0}:=\{n\in \mathbb{N}:z_{0}+y_{n}-y_{0}\in \sum_{i=1}^kT\}$ is
infinite. Moreover, by hypothesis, there exists a constant $M\in \mathbb{R}$
with $f(\sum_{i=1}^kx_{i})\leq \sum_{i=1}^{k}f(x_{i})\leq kM$ for each $x_{1},\ldots ,x_{k}\in T$; so $f(x)\leq kM$ for each $x\in \sum_{i=1}^kT$. Thus, by
Lemma~\ref{1},
\begin{equation*}
n<f(y_{n})\leq f(y_{n}-y_{0}+z_{0})+f(y_{0}-z_{0})\leq kM+f(y_{0}-z_{0})\end{equation*}
for each $n\in \mathbb{N}_{0}$, so $\mathbb{N}_{0}$ is finite, a
contradiction.
\end{proof}

Next we consider some applications of Theorem~\ref{main}.

Since each non--empty open set is shift--compact, we obtain the following
generali\-zation of Theorem~\ref{t1}.

\begin{corollary}
\label{c1} If $X$ is an abelian metric group and $f:X\rightarrow \mathbb{R}$
a subadditive function locally bounded above at some point, then $f$ is
locally bounded at every point of~$X$.
\end{corollary}

The above corollary is also obtained in the case of $\mathbb{R}$ in \cite[Theorem~R]{BinO7} (cf. also \cite[Lemma~4.3]{BinO2}), but the proof there
relies only on group structure, as here.

By Theorem~\ref{HN+HM}, in a complete abelian metric group each universally
Baire non--Haar--meagre set, and also each universally measurable
non--Haar--null set, is shift--compact. Thus we obtain the next result,
generalizing to some extent Theorem~\ref{t2}.

\begin{corollary}
\label{c2} If $X$ is a complete abelian metric group and $f:X\rightarrow
\mathbb{R}$  a~subadditive function bounded above on a set $T\subset X$ with some $k$--fold sum $\sum_{i=1}^kT$  either universally Baire and non--Haar--meagre, or universally measurable and non--Haar--null, then $f$ is locally bounded at each point of~$X$.
\end{corollary}

\section{A connection with generic subadditive functions}

In 2008 N.H. Bingham and A.J. Ostaszewski \cite{BO} (see also \cite{BinO?}, \cite{BinO??} ) introduced the notion of Weak No Trumps functions (called WNT--functions).

\begin{definition}
Let $f:X\rightarrow \mathbb{R}$ be defined on an abelian metric
group\,\footnote{\,In fact, Bingham and Ostaszewski defined a WNT--function in the case $X=\mathbb{R}^{N}$.} and $H^{k}:=f^{-1}(-k,k)$ for $k\in \mathbb{N}$. Call $f$
a \textbf{WNT--function} if for every convergent sequence $(u_{n})_{n\in
\mathbb{N}}$ in $X$ there exist $k\in \mathbb{N}$, an infinite set $\mathbb{M}\subset \mathbb{N}$ and $t\in X$ such that $\{t+u_{m}:m\in \mathbb{M}
\}\subset H^{k}$.
\end{definition}

The more basic No Trumps combinatorial principle, denoted {NT}, refers to a~family of
subsets of reals $\{T_k: k\in\mathbb{N}\}$ (see \cite[Definition~2]{BinO?}).
The function class WNT is one of a hierarchy introduced in \cite{BinO??} and is
so named as it refers to the weakest condition in the hierarchy.

One readily observes the following.

\begin{lemma}
A function $f:X\rightarrow \mathbb{R}$ defined on an abelian metric group $X$
is {\em WNT} if and only if for every sequence $(u_{n})_{n\in \mathbb{N}}$
convergent to $0$ in $X$ there exist $k\in \mathbb{N}$, an~infinite set $\mathbb{M}\subset \mathbb{N}$ and $t\in X$ such that $\{t+u_{m}:m\in \mathbb{M}\}\subset H^{k}$.
\end{lemma}

Now let us present some connections between the WNT--property, boun\-dedness on a
shift--compact set, and local boundedness at a point.

\begin{proposition}
\label{WNT} Let $f:X\rightarrow \mathbb{R}$ be defined on an abelian metric
group. Then the following implications hold:
\begin{itemize}
\item[{(i)}] if $f$ is locally bounded at a point, then $f$ is bounded on a shift--compact set in~$X$;
\item[{(ii)}] if $f$ is bounded on a shift-compact set in $X$, then $f$ is \emph{WNT}.
\end{itemize}
\end{proposition}

\begin{proof}
(i) It is an easy consequence of the fact that open sets are shift--compact.

(ii) Assume that for some shift--compact set $D\subset X$ there exists $k\in
\mathbb{N}$ such that $f(D)\subset (-k,k)$. Since $D$ is shift--compact, for
each sequence $(x_{n})_{n\in \mathbb{N}}$ tending to $0$ in $X$ there are $x_{0}\in X$ and an infinite set $\mathbb{N}_{0}\subset \mathbb{N}$ such that
$f(x_{n}+x_{0})\in f(D)\subset (-k,k)$ for each $n\in \mathbb{N}_{0}$.
Consequently, for each sequence $(x_{n})_{n\in \mathbb{N}}$ convergent to $0$
there are $k\in \mathbb{N}$, $x_{0}\in X$ and an infinite set $\mathbb{N}_{0}\subset \mathbb{N}$ such that $x_{n}+x_{0}\in f^{-1}(-k,k)$ for $n\in
\mathbb{N}_{0}$.  So $f$ is WNT.
\end{proof}

Note that the converse implication to (i) in Proposition \ref{WNT}
does not hold.

\begin{example}
Define a function $g:[0,1)\rightarrow \mathbb{R}$ by
\begin{equation*}
g(x):=\left\{
\begin{array}{ccl}
(-1)^{n}n, &\;\; \mbox{ for} & x={m}/{n}\in \mathbb{Q}\cap (0,1),\;\mbox{where}\;\gcd \,(m,n)=1, \\[1ex]
0, &\;\; \mbox{ for} & x\in \{0\}\cup \lbrack (0,1)\setminus \mathbb{Q}].
\end{array}
\right.
\end{equation*}
First we prove that $g$ is locally unbounded at each point of $[0,1)$.
Indeed, for each $x\in \lbrack 0,1)$ every open neighbourhood $U_{x}\subset
\lbrack 0,1)$ of $x$ contains infinitely many positive rational numbers,
hence $\sup_{t\in U_{x}}|g(t)|=\infty $.

Next, define a function $f:\mathbb{R}\rightarrow \mathbb{R}$ by
\begin{equation*}
f(x):=g(x-[x])\;\;\mbox{ for}\;\;x\in \mathbb{R},
\end{equation*}
with $[x]$ the integer part of~$x\in \mathbb{R}$. Clearly, $f$ is locally
unbounded at each point of~$\mathbb{R}$.

Moreover, by Theorem~\ref{HN+HM}, the set $\mathbb{R}\setminus \mathbb{Q}$
is shift--compact (as it has positive Lebesgue measure) and $f(\mathbb{R}
\setminus\mathbb{Q})=\{0\}$. Thus, according to Proposition~\ref{WNT}~(ii), $f$ is WNT.
\end{example}

Note also that there exists a function $f:\mathbb{R}\to\mathbb{R}$ bounded
above on a~shift--compact set in $\mathbb{R}$ which is not WNT. It means that we are not able to weaken the assumption in Proposition~\ref{WNT}~(ii).

\begin{example}
\label{ex2} Let $f:\mathbb{R}\to\mathbb{R}$ be given by $f(x):=-|g(x)|$ for $x\in\mathbb{R}$ with a~discontinuous additive function $g:\mathbb{R}\to
\mathbb{R}$. Clearly then $f$ is bounded above on $\mathbb{R}$. We have to show yet that $f$ is not
WNT.

Since $g$ is additive and discontinuous, its graph is dense in $\mathbb{R}^2$ (see \cite[Theorem~12.1.2]{Kuczma}). So, there exists a sequence $(u_n)_{n\in\mathbb{N}}$ convergent to $0$ such that $|g(u_n)|>n$ for every $n\in\mathbb{N}$. Fix $k\in\mathbb{N}$ and $t\in \mathbb{R}$. Then
$$
\begin{array}{ll}
\{n\in\mathbb{N}:|f(t+u_n)|<k\}\!\!\!\!&=\{n\in\mathbb{N}:-k-g(t)<g(u_n)<k-g(t)\}\\[1ex]
&\subset \{n\in\mathbb{N}:|g(u_n)|<\max \{|k-g(t)|,|k+g(t)|\}\}.
\end{array}
$$
Hence the set $\{n\in\mathbb{N}:t+u_n\in f^{-1}((-k,k))\}$ is finite, which means that $\{t+u_n:n\in\mathbb{M}\} \not\subset H^k$ for every infinite set $\mathbb{M}\subset \mathbb{N}$.
\end{example}

%From Proposition~\ref{WNT}~(i) and Example~\ref{ex2}, we see that each WNT--function is necessarily bounded above on a shift--compact set, but the
%converse implication need not hold. Consequently, Theorem~\ref{main} generalizes the following result obtained in \cite{BO}.

In view of Example~\ref{ex2} we see that Theorem~\ref{main} can not be derived from the following result obtained in \cite{BO}.

\begin{proposition}
[{\cite[Proposition~1]{BO}}]\label{BO1} If $f:\mathbb{R}^N\to\mathbb{R}$
is a subadditive \emph{WNT}--function, then $f$ is locally bounded at each
point of $\mathbb{R}^N$.
\end{proposition}

We have shown that generally, for every $f:X\rightarrow \mathbb{R}$
defined on an~abelian metric group,
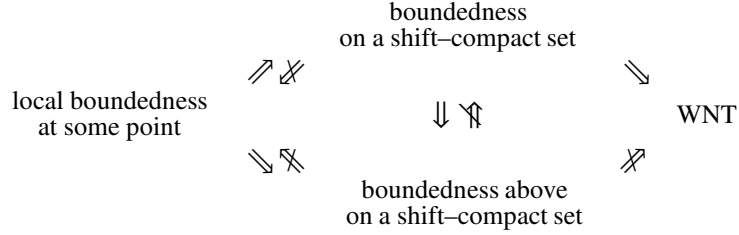
\begin{equation*}
\begin{array}{ccccc}
& \!\!\!&\!\!\! \begin{array}{c}
\mbox{boundedness}\\[-.5ex]
\mbox{on a shift--compact set}
\end{array} & & \\
& \Nearrow\,\backslash\hspace{-.35cm}\Swarrow & & \Searrow & \\
\begin{array}{c}
\mbox{local boundedness} \\[-.5ex]
\mbox{at some point}
\end{array}& &\!\!\! {\Large\mbox{$\Downarrow$~$\diagdown\hspace{-.33cm}\Uparrow$}} & & \mbox{WNT}\\
& \Searrow\,\slash\hspace{-.35cm}\Nwarrow & & \backslash\hspace{-.35cm}\Nearrow & \\
& & \begin{array}{c}
\mbox{boundedness above}\\[-.5ex]
\mbox{on a shift--compact set}
\end{array}& & \\[5pt]
%\begin{array}{c}
%\;\Longrightarrow \;\\[-1ex]
%\; \slash\hspace{-5mm} \Longleftarrow
%\end{array} &
%\begin{array}{c}
%\mbox{boundedness}\\
%\mbox{on a shift--compact set}
%\end{array}&  \;\Longrightarrow \;&\mbox{WNT}\\
%& & \Downarrow\,\,\not\Uparrow & \backslash\hspace{-.35cm}\Nearrow & \\
%& & \begin{array}{c}
%\mbox{boundedness above}\\
%\mbox{on a shift--compact set}
%\end{array}&
\end{array}
\end{equation*}

\begin{problem}
Does every WNT function have to be bounded (above) on a shift--compact set?
\end{problem}

Nevertheless, combining Theorem \ref{main} and Propositions \ref{WNT} and \ref{BO1},
we deduce that the situation is completely different in the class of {\em subadditive functions}.

\begin{corollary}
Let $f:X\rightarrow \mathbb{R}$ be a subadditive function defined on an abelian metric
group. The following conditions are equivalent:
\begin{itemize}
\item $f$ is locally bounded at some point;
\item $f$ is WNT;
\item $f$ is bounded above on a shift--compact set;
\item $f$ is bounded on a shift--compact set.
\end{itemize}
\end{corollary}

\section{Concluding remarks}

~~~~\textbf{4.1.} This joint paper arose out of a newfound interest in shift--compactness.
The concept was isolated a decade ago in establishing a common proof for the
two known `generic cases' of measure and category of the Uniform Convergence
Theorem for regularly varying functions (for background see \cite{BinGT} and
\cite{BinO?}). Most recently one of the current authors in collaboration
initially with T. Banakh inititated the study of a common generalization
of Christensen's notion of Haar--null sets and Darji's notion of Haar--meagre
sets by replacing the relevant $\sigma $--ideals by other ideals $\mathcal{I}$
and working in the general context of abelian metric groups. This yielded
Haar--$\mathcal{I}$ sets, and led in particular to the notion of null--finite
sets (corresponding to the ideal of finite subsets). A null--finite set is
simply a non--shift--compact set and this explains the recent resurgence of
interest in shift--compactness in metric groups, and the new results as in
\cite{MilMO}. Results on subadditivity using shift--compactness have been studied
in various publications in the last decade typically in Euclidean contexts,
so it seemed natural to collect together the known results from the widely
scattered literature, in the more natural contexts here of groups or linear
spaces.
\vspace{0.2cm}

\textbf{4.2.} Use of large sets in the contexts of additivity, subadditivity, and
convexity may be traced back to the work of R.~Ger and M.~Kuczma on `test sets'
(the Ger--Kuczma classes of test sets $\mathfrak{A,B,C}$ \cite[Chapters~9,10]{Kuczma}), the idea being that a property holding on a test set would
automatically imply a related property globally, as in the well--known
related case of \textit{automatic continuity} (for which see H.G. Dales
\cite{Dal}, \cite{Dal1} and \cite{BacBCDT}). Such ideas were also pursued by Z.~Kominek on the basis of result by
B.~Jones (test set there being capable of spanning the reals, as with
Hamel bases but not coincidental with these; for recent work on Hamel bases
see e.g. \cite{CieP1}, \cite{CieP2}, \cite{DorFN}). An alternative approach to test sets was
developed also in \cite{BinO-10} with connections to uniformity results in the
theory of regular variation.

For further thematic uses of large sets (in connection with regularity, i.e. `smoothness', properties of functions) see the extended arXiv version of this paper.

\vspace{0.2cm}

\bigskip

\textbf{Postscript.} With sadness we dedicate this paper to the memory of Harry I.~Miller, friend and collaborator, recently passed away -- a latter--day pioneer of shift--compactness (see \cite{Mil}).

\end{document}